\documentclass[12pt]{amsart}
\usepackage{amsmath}
\usepackage{amssymb}
\usepackage{amscd}
\usepackage{pstricks}
\usepackage{color}
\usepackage{mathpazo}
\usepackage[dvips]{graphicx}
\usepackage{textcomp}

\theoremstyle{plain}
\newtheorem{theorem}{Theorem}
\newtheorem{lemma}{Lemma}

\newtheorem{corollary}[lemma]{Corollary}

\def\Prob{{\mathbb{P}}}

\title{Weak convergence of random walks, conditioned to stay away}

\author{Zsolt Pajor-Gyulai\and Domokos Sz\'asz}

\begin{document}
\maketitle

\begin{abstract}
Let $\{X_n\}_{n\in\mathbb{N}}$ be a sequence of i.i.d. random variables in $\mathbb{Z}^d$.
Let $S_k=X_1+...+X_k$ and $Y_n(t)$ be the continuous process on $[0,1]$ for which $Y_n(k/n)=S_k/\sqrt{n}$ $k=1,...,n$ and which is linearly interpolated elsewhere. The paper gives a generalization of results of Belkin, \cite{B72} on the weak limit laws of $Y_n(t)$ conditioned to stay away from some small sets. In particular, it is shown that the diffusive limit of the random walk meander on $\mathbb Z^d: d\ge 2$ is the Brownian motion.

\end{abstract}

\section{Introduction}
In his \cite{B70} paper, Belkin examined the asymptotic effect of conditioning on the asymptotic behavior of a random walk. In his case conditioning meant that the random walk was supposed to avoid a certain finite subset of $\mathbb{Z^d}$.  By using characteristic functions, he, for instance, showed that if - in one dimension - the original limit law is normal, then the conditioned walk approaches a two sided Rayleigh distribution. He also showed that - in two dimension - the conditioning has no effect on the limit law. Later, he strengthened his results by proving the corresponding weak invariance theorems (\cite{B72}). His method was, however, pretty technical. Bolthausen offered a more elegant technique when he asked what is the limit law - in the diffusive scaling - of a random walk of finite variance on $\mathbb Z$ conditioned to stay positive (\cite{B76}); he found that the limiting process is the so-called Brownian meander. An interesting consequence of our result is that, in dimension $d\ge 2$, the diffusive limit of the random walk meander is the Brownian motion.

The main goal of this paper is to prove that, in general, conditioning has no effect on the limit distribution if the forbidden subset has zero measure with respect to the unconditioned limiting distribution. Our method is based on Bolthausen's functional approach.

The key observation in his proof is that a random time, being not a stopping time, nevertheless behaves like a stopping time. With an appropriate modification of the definition of Bolthausen's stopping time his basic equation still remains valid,  cf. the Lemma \ref{thm:eq_distr} of this paper. It is worth noting that our proof is actually simpler than that of Bolthausen since, in particular, we also use the results in \cite{W70}.

Our motivation for treating this problem was that, in \cite{P-GySz10b} (cf. \cite{P-GySz10a}),
we needed a generalization of Corollary 2 (to continuous time random walks with internal states) for
describing the diffusive limit of a stochastic model of two Lorentz disks. Having made a research in the literature
we were surprised to learn that even for Corollary 1 we could not localize any reference.

\section{Notations and result}
Let $C^d[0,n]$ be the set of continuous functions from the interval $[0,n]$ to $\mathbb{R}^d$ and let $\rho_n$ be the usual supremum metric on $C^d[0,n]$:
\[
\rho_n(f,g)=\sup_{0\leq t\leq n}|f(t)-g(t)|.
\]
We will also use the space $C^d[0,\infty)$ endowed with the metric
\[
\rho(f,g)=\sum_{n=1}^{\infty}2^{-n}\frac{\rho_n(f,g)}{1+\rho_n(f,g)}.
\]
It was shown by Whitt (\cite{W70}) that convergence (of the natural projections) in $(C^d[0,n],\rho_n)$ for every $n$ and convergence in $(C^d[0,\infty),\rho)$ are equivalent.

Let $\{X_n\}_{n\in\mathbb{N}}$ be a sequence of i.i.d. random variables in $\mathbb{Z}^d$. The random walk generated by the partial sums is $S_k=X_1+...+X_k$ and denote by $S^{int}_t$ its continuous, linearly interpolated trajectory. Denote finally by $Y_n(t)= S^{int}_t/\sqrt{n}\ ( t \in [0,1])$ its diffusively scaled variant.  Let $Q_n(B)=\Prob(Y_n\in B)$ for every Borel subset $B$ of $C^d[0,1]$.

Let $A$ be a linear subspace of $\mathbb{R}^d$ and
\[
\tilde{C}_A=\{f\in C^d[0,1]| f(t)\notin A\quad t\in(0,1]\}.
\]
Also define the conditioned process $\tilde{Y}_n$ by
\[
\Prob(\tilde{Y}_n\in B)=Q_n(B|\tilde{C}_A).
\]
We will need the above processes extended to the whole half line, too. Define  $\overline{Y}_n$ as the continuous process for which $\overline{Y}_n(k/n)=S_k/\sqrt{n}$ for $k\in\mathbb{N}$ and linearly interpolated elsewhere. Note $Y_n=\overline{Y}_n|_{[0,1]}$. Let $\overline{Q}_n(B)=\Prob(\overline{Y}_n\in B)$ for every Borel subset $B$ of $C^d[0,\infty)$ and let $\Pi_1$ denote the natural projection $C^d[0,\infty) \to C^d[0,1]$. Note that $Q_n=\overline{Q}_n\Pi_1^{-1}$.

Assume that $\overline{Y}_n\Rightarrow \overline{Y}_{\infty}$ in $(C^d[0,\infty),\rho)$ and let $\overline{P}(B)=\Prob(\overline{Y}_{\infty}\in B)$. This implies that $Y_n\Rightarrow Y_{\infty}=\overline{Y}_{\infty}|_{[0,1]}$, so let $P=\overline{P}\Pi_{1}^{-1}$ be the measure generated by $Y_{\infty}$.  Our result is

\begin{theorem}\label{thm:negligible}
If  $P(\tilde{C}_A)=1$, then $\tilde{Y}_n\Rightarrow Y_{\infty}$ in $(C^d[0,1],\rho_1)$.
\end{theorem}

Consider the closed subspace $C_0^d[0,\infty)$ which consists of the continuous functions in $C^d[0,\infty)$ with $f(0)=0$. Let $\overline{Q}_{n,0}$ and $\overline{P}_0$ denote the corresponding restricted measures, i.e. for $B_0\in\mathcal{B}(C_0^d[0,\infty))$
\[
\overline{Q}_{n,0}(B_0)=\overline{Q}_n(B_0).
\]
Also introduce $Q_{n,0}$ and $P_0$ in the same fashion. Since $C_0^d[0,\infty)$ is the support of both $\overline{Q}_n$ and $\overline{P}$,
\[
\overline{Q}_n(B)=\overline{Q}_{n,0}(B\cap C_0^d[0,\infty))
\]
and similarly for $\overline{P}$. Thus it suffices to conduct the proof using this smaller space. The reason for doing so will  become apparent in the next section. $\Pi_{1,0}$ will denote the natural projection from $C_0^d[0,\infty)$ to $C_0^d[0,1]$, i. e. $\Pi_{1,0}=\Pi_1|_{C_0^d[0,\infty)}$.

We conclude this section with some corollaries of Theorem \ref{thm:negligible}.

\begin{corollary}
In  dimension $d\ge 2$, a zero mean, finite variance random walker whose interpolated trajectory is conditioned to avoid returning to the origin converges weakly to a $d$-dimensional Brownian motion.
\end{corollary}

The reader might note that in the above corollary jumps like $(-1,-1)\to (1,1)$ are not allowed, a fortuitous consequence of interpolation. (For instance, one can change the interpolation so that when the interpolated trajectory were to hit the origin, it goes around it on a circle of infinitesimal radius. This clearly does not change the limiting process and does not effect the validity of the proof of Theorem \ref{thm:negligible}.) The natural way is to take conditioning on the random walk (the random walk meander) and then to take our favorite interpolation. Thus we have

\begin{corollary}
In dimension $d\ge 2$, the limit (in the diffusive scaling) of a zero mean, finite variance random walk, conditioned to avoid returning to the origin, converges weakly to a $d$-dimensional Brownian motion.
\end{corollary}

Obviously, the previous trick works when $dim(A)\leq d-1$. We state our last two corollaries in this spirit.
\begin{corollary}
We also have that - for $d\ge 2$ - two independent, zero mean, finite variance random walkers both starting from the origin and conditioned not to meet after they depart will converge to the product of their independent limits. To see this, let $S_n^{(1)}$, $S_n^{(2)}$ denote the two independent random walkers, while $S_n=(S_n^{(1)}, S_n^{(2)})$ stands for the $2d$-dimensional composite walk. Now apply Theorem \ref{thm:negligible} and  Donsker's theorem (more precisely its multidimensional generalization) to $S_n$ with e.g. $A=\{(x,y,w,z)|x=w, y=z\}$ in $d=2$.
\end{corollary}

\begin{corollary}
Consider a $d(\ge2)$-dimensional random walk with zero mean and finite variance and let $A'$ be the subset of any $(d-2)$-dimensional subspace. By Donsker's theorem again, the unconditional limit is a $d$-dimensional Brownian motion. By applying Theorem \ref{thm:negligible} and the previous observation to the linear span of $A'$, we see that the conditional limit is the same.
\end{corollary}

\section{Proof}
We will prove the assertion by establishing a connection between $\{Y_n(t): t\in[0,1]\}$ and an appropriately chosen random segment of $\overline{Y}_n(t)$.

Define the functional $T: C_0^d[0,\infty)\to\mathbb{R}_+$ by
\[
T(f)=\inf\{t: f(t)\in A, f(u)\notin A,  t< u\leq t+1\}\quad (\inf\emptyset=\infty).
\]
Similarly to the analogous result in \cite{B76}, it is not hard to show that this functional is measurable. Then $\overline{P}_0(T=0)=\overline{P}_0\Pi_1^{-1}(\tilde{C}_A)=P(\tilde{C}_A)=1$ since $A$ contains the origin. It is not hard to see that the discontinuity set of $T$ in $C_0^d[0,\infty)$ is $D_T=T^{-1}(0,\infty)$ which has $\overline{P}_0$-measure 0.

Also define the mapping $\Phi: C_0^d[0,\infty)\to C_0^d[0,1]$ by
\[
(\Phi(f))(t)=f(T(f)+t)-f(T(f))
\]
and note that $\overline{P}_0(\Phi(f)=f|_{[0,1]})=1$ and that $\Phi$ is continuous $\overline{P}_0$-a.e.

Now turn to the walk and denote
\[
T_n=\inf\{k: S_k\in A,S^{int}_{k+t}\notin A\quad t\in (0,n]\}.
\]

Note that $\Prob(T_n<\infty)=1$ and set $Z_k=S_{T_n+k}-S_{T_n}$. The key element in the sequel is a
form of Bolthausen's equation, appropriate for our purpose. It says that - in the same way as in his case - though $T_n$ is not a stopping rule nevertheless it acts as a stopping rule.

\begin{lemma}\label{thm:eq_distr}
For each $B_1,...,B_n\in \mathcal{B}(\mathbb{R}^d)$,
\[
\Prob(S_k\in B_k, k=1,...,n| S_t^{int}\notin A, t\in(0,n])=\Prob(Z_k\in B_k, k=1,...,n)
\]
\end{lemma}

\begin{proof}
Since $A$ is a subspace,  $A\cap\mathbb{Z}^d$ is a sublattice and the walk essentially starts over after hitting it. Thus, $[S_{T_n+k}-S_{T_n}, k=1,..,n]$ is independent of $S_{T_n}$ and of the past of the process and has the distribution of $[S_k, k=1,..,n|S_t^{int}\notin A , t\in(0,n]]$.
\end{proof}

By the above lemma,
\[
Q_{n,0}(.|\tilde{C}_A)=\overline{Q}_{n,0}\Pi_{1,0}^{-1}(.|\tilde{C}_A)=\overline{Q}_{n,0}\Phi^{-1}(.).
\]

By assumption, $\overline{Q}_{n,0}\Rightarrow \overline{P}_0$ so by the virtue of the continuous mapping theorem (Theorem 5.1 in \cite{B68}), this converges weakly to $\overline{P}_0\Phi^{-1}$ in $(C^d[0,1],\rho_1)$. This limit is nothing else but the measure generated by $\overline{Y}_{\infty}(T+.)-\overline{Y}_{\infty}(T)=Y_{\infty}(.)$ a.e. since $\overline{P}_0(T=0)=1$.

As it was mentioned, it is trivial to extend the result and show $Q_n(.|\tilde{C}_A)\Rightarrow P$.

\section{Remarks}
\begin{enumerate}
\item Although it does not follow from the above proof, it is clear on an intuitive basis that Theorem \ref{thm:negligible} should hold for a finite union of such subspaces.
\item Our result can be carried over easily to continuous time random walks. If $S_t$ is the position of the continuous time random walker, then $Y_n(t)=S_{nt}/{\sqrt{n}}$ and one should use the space $D[0,\infty)$ endowed with the Skorohod topology (see \cite{B68} and \cite{L73}). Also one must replace $\tilde{C}_A$ with
\[
\tilde{C}_{A}^{(cont)}=\{f\in D[0,1]| f(t)\notin A\quad t\in[\xi,1]\}
\]
where $\xi$ is the time of the first jump. Then without any difficulty (at least in principle), one can prove the result analogous to Theorem \ref{thm:negligible}.
\item The technique of this proof is a very powerful one. Suppose we have a conditioned measure on $C^d[0,1]$. Let $T$ denote the functional on $C^d[0,\infty]$ which gives the random time when the condition first happens for $T+t, t\in[0,1]$. Also let $T_n$ denote the time after which the linearly interpolated walk divided by $\sqrt{n}$ satisfies the condition. If the analog of Lemma \ref{thm:eq_distr} can be proved, then the limiting process is $[Y_{\infty}(T+t)-Y_{\infty}(T), t\in[0,1]]$.
\end{enumerate}

\noindent{\bf Acknowledgement.} The authors are most indebted to P\'eter N\'andori for his valuable comments and criticism upon an earlier version. They express their sincere gratitude to Greg Lawler  who also suggested an
alternative idea for proving Corollary 1 and to Erwin Bolthausen and B\'alint T\'oth for their useful
remarks.


\begin{thebibliography}{}

\bibitem{B70}[B 70] B. Belkin \textit{A limit theorem for conditioned recurrent random walk attracted to a stable law} The Annals of Mathematical Statistics, {\bf{41}}, No. 1, 146-163, 1970

\bibitem{B72}[B 72] B. Belkin \textit{An invariance principle for conditioned recurrent random walk attracted to a stable law} Probability Theory and Related Fields, {\bf{21}}, No. 1, 45-64, 1972

\bibitem{B68}[B 68] P. Billingsley \textit{Convergence of probability measures} Wiley, New York, 1968

\bibitem{B76}[B 76] E. Bolthausen \textit{On a functional central limit theorem for random walks conditioned to stay positive} The Annals of Probability, {\bf{4}}, No. 3, 480-485, 1976

\bibitem{L73}[L 73] T. Lindvall \textit{Weak Convergence of Probability Measures and Random Functions in the Function Space $D[0,\infty)$} Journal of Applied Probability {\bf 10}, No. 1, 109-121, 1973

\bibitem{P-GySz10a}[P-GySz 10a] Zs. Pajor-Gyulai, D. Sz\'asz. \textit{Energy Transfer and Joint diffusion},
Proc. XVIth International Congress on Mathematical Physics. ed. P. Exner. World Scientific, 328-332, 2010.

\bibitem{P-GySz10b}[P-GySz 10b] Zs. Pajor-Gyulai, D. Sz\'asz. \textit{Energy Transfer and Joint diffusion},
submitted to Communications in Mathematical Physics

\bibitem{W70}[W 70] W. Whitt \textit{Weak Convergence of Probability Measures on the Function Space $C[ 0, \infty)$} The Annals of Mathematical Statistics {\bf{41}}, No. 3, 939-944 , 1970

\end{thebibliography}
\end{document}